\documentclass[12pt,a4paper]{article}

\usepackage{graphicx,hyperref,enumitem}

\usepackage{amsmath,amsthm,amssymb,color}
\usepackage[noadjust,sort]{cite}
\usepackage{hyperref}

\usepackage{tikz}
\usetikzlibrary{calc}

\usepackage[font=small,labelfont=bf]{caption}
\normalsize

\newtheorem{thm}{Theorem}[section]
\newtheorem{cor}[thm]{Corollary}
\newtheorem{lemma}[thm]{Lemma}
\newtheorem{conj}[thm]{Conjecture}

\numberwithin{equation}{section}

\newcommand{\hmax}{h_{\mathrm{max}}}
\newcommand{\lmax}{\ell_{\mathrm{max}}}

\newcommand{\st}{\mathrel{:}}
\let\le=\leqslant
\let\ge=\geqslant

\newcommand{\EO}{\operatorname{EO}}

\newcommand{\class}{\operatorname{\mathcal{C}}}

\let\originalleft\left
\let\originalright\right
\renewcommand{\left}{\mathopen{}\mathclose\bgroup\originalleft}
\renewcommand{\right}{\aftergroup\egroup\originalright}

\newcommand{\abs}[1]{\lvert#1\rvert} \let\card=\abs

\renewcommand{\dfrac}[2]{\lower0.12ex\hbox{\large$\textstyle\frac{#1}{#2}$}}
\newcommand{\Dfrac}[2]{\raise0.05ex\hbox{\small$\displaystyle\frac{#1}{#2}$}}

\newcommand{\calE}{\mathcal{E}}

\newcommand{\calP}{\mathcal{P}}

\newcommand{\Pvec}{\boldsymbol{P}}

\newcommand{\mvec}{\boldsymbol{m}}

\newcommand{\calT}{\mathcal{T}}

\newcommand{\paul}{\widehat\rho}
\newcommand{\eo}{\rho}

\newcommand{\E}{\operatorname{\mathbb{E}}}
\renewcommand{\Pr}{\operatorname{\mathbb{P}}}

\newcommand{\cart}{\mathbin{\raise0.15ex\hbox{$\scriptstyle\square$}}}
\newcommand{\Be}{\operatorname{Be}}

\begin{document}

\title{On Pauling's residual entropy estimate for regular graphs with growing degree}
\author{
Mahdieh Hasheminezhad\thanks{Supported by a MATRIX-Simons travel grant}\\
\small Department of Computer Science\\[-0.5ex]
\small Yazd University \\[-0.5ex]
\small Yazd, Iran\\[-0.5ex]
\small\tt hasheminezhad@yazd.ac.ir
\and
Mikhail Isaev\thanks{Supported by Australian Research Council grant DP250101611} \\
\small School of Mathematics and Statistics \\[-0.5ex]
\small UNSW Sydney \\[-0.5ex]
\small Sydney NSW 2052, Australia \\[-0.5ex]
\small\tt m.isaev@unsw.edu.au
\and
Brendan D. McKay\footnotemark[2] \\
\small School of Computing \\[-0.5ex]
\small Australian National University  \\[-0.5ex]
\small Canberra ACT 2601, Australia \\[-0.5ex]
\small\tt brendan.mckay@anu.edu.au
\and
Rui-Ray Zhang\thanks{Supported by an Alf van der Poorten Travelling Fellowship} \\
\small Brisbane QLD 4109, Australia \\[-0.5ex]
\small\tt rui.ray.zhang@hotmail.com
}
\date{}

\maketitle
\begin{abstract}
  In 1935, Pauling proposed an estimate for the number of
  Eulerian orientations of a graph in the context of the
  theoretical behaviour of water ice.
  The logarithm of the number of Eulerian orientations,
  normalised by the number of vertices, is called the
  residual entropy.
  In an earlier paper, we conjectured that the residual
  entropy of a sequence of regular graphs of increasing
  degree was asymptotically equal to Pauling's estimate.
  Here we prove the conjecture under constraints on
  the number of short circuits. These constraints hold
  under weak eigenvalue conditions and apply to sequences
  of increasing girth and repeated
  Cartesian products such as hypercubes.
\end{abstract}


\section{Introduction}

Let $G$ be a $d$-regular graph on $n$ vertices
with even $d$.
Let $\EO(G)$ denote the number of Eulerian orientations of~$G$,
that is 
orientations of all the edges of $G$ such that every vertex has the same in-degree 
as its out-degree.
We consider the logarithm of the number of Eulerian orientations of $G$ normalised
by the number of vertices
\begin{equation*}
\eo(G) := \dfrac{1}{n} \log \EO(G).
\end{equation*}
The quantity $\eo(G)$ is known as the \textit{residual entropy} of ice-type models in statistical physics.
Determining the asymptotic behaviour of $\eo(G)$ as $n\to\infty$ is a key question in the area, see for example
\cite[Chapter 8]{baxter1969f}
and \cite{lieb1972two}. 
In particular, the value is known for the square lattice~\cite{lieb1967residual},
the triangular lattice~\cite{baxter1969f}, and the hexagonal ice monolayer~\cite{LieHexagonalMonolayer}.

Around 90 years ago, Pauling \cite{pauling1935structure} proposed the best-known heuristic estimate for~$\eo(G)$.
Independently orient each edge of $G$ at random. The probability of the event $\calE_i$ that vertex $i \in [n]$ has in-degree equal to out-degree is
$2^{-d}\binom{d}{d/2}$,
where $[n]$ denotes the set $\{ 1, 2,\ldots, n\}$.
Then
\[
    \EO(G) = 2^{nd/2} \Pr\biggl(\,
    \bigcap_{i\in[n]}\calE_i\biggr).
\]
If all the events $\{ \calE_i: i \in[n]\}$ were independent, then we would get 
 \begin{equation*} 
    \paul(G) 
= \dfrac{1}{n} \log \biggl(2^{nd/2} \prod_{i \in [n]} \Pr(\calE_i)\biggr)
  = \log \binom{d}{d/2} - \Dfrac{d}{2} \log 2. 
  \end{equation*}
The heuristic estimate $\paul(G)$ is known as 
the Pauling's residual entropy estimate \cite{pauling1935structure}.
Lieb and Wu~\cite{lieb1972two}
showed that, for any graph $G$ with even degrees,  \[
    \eo(G) \ge \paul(G).
\]

In this paper, we address the following conjecture.
All asymptotic notation is with respect to a
sequence $\{G(n)\}$ of graphs of increasing order,
and $G(n)$ is assumed to have order~$n$.
\begin{conj}[Special case of {\cite[Conjecture 2.2]{IMZ1}} for regular graphs]\label{Conj}
    If $G=G(n)$ is a sequence of $d$-regular graphs with even $d=d(n)\to\infty$ as $n\to\infty$, then 
    \[
            \eo(G) = \paul(G) +o(1).
    \]
\end{conj}

The formula in Isaev, McKay, and Zhang \cite[Theorem 1.1]{IMZ} confirms Conjecture \ref{Conj} for graphs
with good expansion properties and $d\ge \log^8 n$.
Later, in 
\cite[Theorem 2.8]{IMZ1}, they showed that it is also true for random graphs with high probability.

We establish Conjecture~\ref{Conj} for graphs containing not very many short closed trails. In particular, this yields Conjecture~\ref{Conj} for graphs of growing girth and repeated
Cartesian products.
 

\section{Main results}

A \textit{closed trail} in a graph $G$ is a closed walk which uses no edge more than once.
Two closed trails that differ only in the starting point or the direction of traversal are considered the same.
Let $c_{\ell}(G)$ denote the number of closed trails of length $\ell$ in graph $G$. 
Our main result confirms 
Conjecture \ref{Conj} under some weak assumptions on the growth of $c_\ell$ for small $\ell$.

\begin{thm}\label{T:main}
Let  $G = G(n)$ be a regular graph with even degree $d=d(n) \rightarrow \infty$.
Suppose that there is  $\lmax=\omega(\log d)$, 
and some fixed constant $C>0$ such that  
\[
c_\ell(G) \le C e^{-(\ell+1)}d^{\ell-1}n \qquad
\text{for all $3\le \ell\le\lmax$.}
\]
Then Conjecture~\ref{Conj} holds.
\end{thm}

We prove Theorem \ref{T:main}  in Section~\ref{S:proof_idea} based on  Lemma \ref{bigtrails} and Lemma \ref{smalltrails} established in the later sections.
We also establish the following three corollaries, whose proof appears
in Section~\ref{S:corollaries}.
By the ``eigenvalues of $G$'' we mean the eigenvalues of the adjacency
matrix of~$G$.

\begin{cor}\label{C:spectral}
  If $G=G(n)$ is a $d$-regular graph with $d\to\infty$, and 
  at most $nd^{-\omega(1)}$ eigenvalues of $G$ lie
  outside $[-d^{1-\delta},d^{1-\delta}]$ for some constant $\delta>0$,
  then Conjecture~\ref{Conj} holds.
\end{cor}

Las Vergnas~\cite{LasVergnas1990} proved that Conjecture~\ref{Conj}
holds if the girth grows faster than $\log d$.
The following shows that increasing girth is enough.

\begin{cor}\label{C:spectral2}
    If $G=G(n)$ is a sequence of $d$-regular graphs with growing girth
    $g\to\infty$,
    then Conjecture~\ref{Conj} holds.
    (Note that $d\to\infty$ is not required for
    this corollary.)
\end{cor}


\begin{cor}\label{C:product}
Let $G_t = H^{(t)}_1 \cart \cdots \cart H^{(t)}_t$
be the Cartesian product of non-empty graphs
$\{ H^{(t)}_i: i \in [t]\}$, where $H^{(t)}_i$ is
$h^{(t)}_i$-regular and $d_t=\sum_{i=1}^t h^{(t)}_i$ is even.
Then Conjecture~\ref{Conj} holds as $t\to\infty$ if
\[
   \sum_{i=1}^t (h^{(t)}_i)^2 = O(d_t^{2-\delta})
\]
for some $\delta>0$.
In particular, if $1\le h^{(t)}_i\le\hmax$ for a 
global constant $\hmax$, then Conjecture~\ref{Conj} holds.
Similarly, Conjecture~\ref{Conj} holds if, for
each~$t$, $h^{(t)}_1=\cdots=h^{(t)}_t=h(t)$, where
$h(t)\ge 1$ is bounded by a polynomial in~$t$.
\end{cor}

Corollary \ref{C:product} implies that 
Conjecture \ref{Conj} holds for the $d$-dimensional hypercube $Q_d$ with even~$d$.
Another interesting example is the product of cycles whose length
increases with $t$, which locally converges to a square
lattice of increasing dimension.
See \cite{IMZ1} for dimensions 2 and~3.

\subsection{Proof of Theorem \ref{T:main}}\label{S:proof_idea}

First, we present a convenient formula for the difference between the residual entropy $\rho(G)$ and Pauling's estimate $\paul(G)$. 
A \textit{pairing} of an Eulerian graph $G$ is a division
of the edges incident to each vertex into unordered pairs.
A closed trail $T$ is \textit{induced} by the pairing if each consecutive pair of edges in $T$ are
paired at their common vertex.
An \textit{Eulerian partition} is a partition of the edges into closed trails.
An example is shown in Figure~\ref{fig:pairing}.
Let $\calP(G)$ denote the set of all Eulerian partitions, and note that
\[
   \card{\calP(G)} = \biggl( \frac{d!}{(d/2)!\,2^{d/2}}\biggr)^{\!n},
\]
since each partition is uniquely described by a pairing of the edges at each vertex. For an Eulerian partition $P\in\calP(G)$, let $\calT(P)$ denote the set of closed trails it induces.

\begin{figure}[ht]
    \centering
    \includegraphics[width=0.5\linewidth]{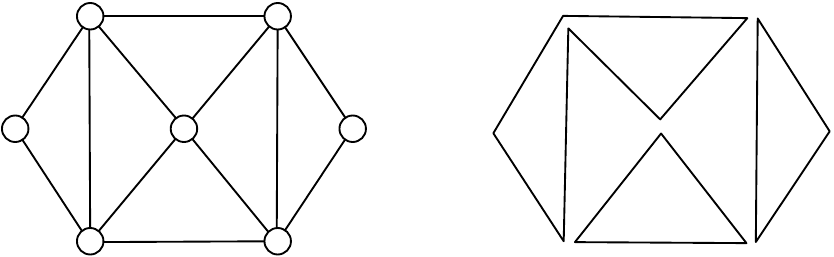}
    \caption{A graph $G$ and an Eulerian partition $P \in \calP(G)$.}
    \label{fig:pairing}
\end{figure}

Let $\Pvec$ be a uniform random element
from $\calP(G)$.
 We have that
\[
\eo(G) = \paul(G) + \Dfrac1n \log \E 2^{|\calT(\Pvec)|};
\]
 see, for example,   \cite[Theorem 5.1]{IMZ1}.
 Thus, Conjecture~\ref{Conj} is equivalent to 
 \[
  \E 2^{|\calT(\Pvec)|} = e^{o(n)}.
 \]

Define 
\begin{align}
\label{eq:kDef}
k = \lfloor \min\{ \lmax/2 , \log^2 d\}\rfloor.    
\end{align}
Let $S_k(\Pvec)$ denote the number of trails from $\calT(\Pvec)$ that contain at most $k$ distinct vertices,
and let $L_k(\Pvec)$ denote the number of other trails from $\calT(\Pvec)$. Using H\"older's inequality, we can bound
\[
   \E 2^{|\calT(\Pvec)|} = \E 2^{L_k(\Pvec) + S_k(\Pvec) }
 \le \bigl(\E 2^{\frac{7}{2}L_k(\Pvec)}\bigr)^{2/7}  \bigl(\E 2^{\frac75 S_k(\Pvec)}\bigr)^{5/7}. 
\]

We will show that  $\E 2^{\frac72 L_k(\Pvec)} = e^{o(n)}$ in Lemma~\ref{bigtrails}
and $\E 2^{\frac75 S_k(\Pvec)} = e^{o(n)}$ 
in Lemma~\ref{smalltrails}.
This will complete the proof of Theorem \ref{T:main}.

\subsection{Proofs of the corollaries}\label{S:corollaries}

\begin{proof}[Proof of Corollary \ref{C:spectral}]
  Suppose the number of eigenvalues outside of $[-d^{1-\delta},d^{1-\delta}]$
  is $nd^{-f(n)}$.
  We will show that the conditions of Theorem~\ref{T:main} hold
  with $C=e^{1+2/\delta}$ and $\lmax=\frac12 f(n)\log d$.

  First consider closed trails of length $\ell\le 2/\delta$.
  Since the final step in a closed trail is determined, we have the
  uniform bound $c_\ell \le n d^{\ell-1}\le Cd^{\ell-1}e^{-(\ell+1)}n$.

  Now consider trails of length $\ell\in[2/\delta,\lmax]$.
  Let $\{\lambda_i\}$ be the eigenvalues of~$G$.
  Since the number of closed trails of length $\ell$ is bounded by the
  number of closed walks of length~$\ell$, we have
  \begin{align*}
   c_\ell \le \sum_{i\in[n]} \lambda_i^\ell
          &= \sum_{i\in[n]:\abs{\lambda_i}\le d^{1-\delta}}
             \lambda_i^\ell
            + \sum_{i\in[n]:\abs{\lambda_i}> d^{1-\delta}}
             \lambda_i^\ell \\
          &\le nd^{\ell-1}\bigl(d^{1-\delta\ell} + d^{1-f(n)}\bigr).
\end{align*}
Since $d\to\infty$, we can assume $d>e^{2/\delta}$.
Note that $Ce^{-(\ell+1)}=e^{2/\delta-\ell}$. We have
\begin{align*}
 d^{1-\delta\ell} &= e^{(1-\delta\ell)\log d} \le e^{2/\delta-2\ell}
                         \le \dfrac12 e^{2/\delta-\ell} \\
 d^{1-f(n)} &= e^{(1-f(n))\log d} \le \dfrac12 e^{2/\delta-(1/2) f(n)\log d}
                          = \dfrac12 e^{2/\delta-\lmax}
                          \le \dfrac12 e^{2/\delta-\ell}.
\end{align*}
This completes the proof.
\end{proof}

\begin{proof}[Proof of Corollary \ref{C:spectral2}]
  If $d=O(1)$, then Conjecture~\ref{Conj} holds since the number of closed
  trails induced by a pairing is at most $nd/(2g)=o(n)$.

  Suppose instead that $d\to\infty$.
  Let $\ell=\ell(n)$ be an increasing even number less than the girth.
  The number of closed walks in $G$ of length $\ell$ is $n$ times the
  $\ell$-th moment of a distribution with support $[-2\sqrt{d-1},2\sqrt{d-1}]$,
  and so is at most $2^\ell(d-1)^{\ell/2}n$, see~\cite{mckayeigens}.
  This allows at most $2^\ell d^{-\ell/6}n=nd^{-\omega(1)}$ eigenvalues with
  absolute value greater than $d^{2/3}$.
  Now apply Corollary~\ref{C:spectral}.
\end{proof}

\begin{proof}[Proof of Corollary \ref{C:product}]
Note that $d_t$ is the degree of $G_t$.
Let $X(H^{(t)}_i)$ be the discrete distribution on
$[-h^{(t)}_i,h^{(t)}_i]$ whose weight at each
$x$ is proportional to the multiplicity of $x$ as an
eigenvalue of $H^{(t)}_i$.
Then $X(G_t)$ is the independent sum
$\sum_{i=1}^t X(H^{(t)}_i)$, see~\cite{Merris1998}.
By Hoeffding's inequality \cite{Hoeffding},
\[
   \Pr\bigl( \abs{X(G_t)} \ge d_t^{1-\delta/4}\bigr) \le
   2\exp\biggl(-\frac{d_t^{2-\delta/2}}
                {2\sum_{i=1}^t (h^{(t)}_i)^2}\biggr)
   = 2\exp\bigl(-\Omega(d_t^{\delta/2})\bigr)
   = d_t^{-\omega(1)}.
\]
This completes the proof of the first part on applying Corollary~\ref{C:spectral}.
For the second part, note that $ \sum_{i=1}^t (h^{(t)}_i)^2=O(t)$ and $d_t^{2-\delta}=\Omega(t^{2-\delta})$.
For the third part, if $h(t)\le \alpha t^\beta$ for constants $\alpha,\beta>0$, take $\delta=(1+\beta)^{-1}$.
\end{proof}

\section{Long closed trails}\label{S:long}

In this section, we bound the moment-generating function for closed trails, induced by a uniform random Eulerian partition $\Pvec \in \calP(G)$,
that have $\omega(\log d)$ distinct vertices. 

For even $m$, let $X(m)$ be a random variable (or distribution) which is
the sum of independent Bernoulli random variables with parameters
$1/(2j-1)$ for $j=1,\ldots,m/2$.

\begin{lemma}\label{Zmprops}
  There is a universal constant $B$ such that,
for even $m\ge 2$ and $\lambda\ge 0$,
 \[
 \E e^{\lambda X(m)}\le (Bm)^{(e^{\lambda}-1)/2}.
 \]
\end{lemma}
\begin{proof}
 For a random variable
  $Z_p$ with Bernoulli distribution $\Be(p)$, we have
  $\E e^{\lambda Z_p}=1+(e^\lambda-1) p\le e^{(e^\lambda-1) p}$.
  Therefore,
  \[
      \E e^{\lambda X(m)} = \prod_{i=1}^{m/2} e^{(e^\lambda-1)/(2i-1)}
       \le (Bm)^{(e^{\lambda}-1)/2},
  \]
  where $B$ is such that $\sum_{i=1}^{m/2} \frac{1}{2i-1} 
  \le \dfrac12\log (Bm)$.
\end{proof}
\begin{lemma}\label{onevertex}
  Consider an Eulerian graph with a vertex $v$ of degree~$m$.
  Then, in a random pairing, the distribution of the number of
  distinct induced closed trails that pass through $v$ is~$X(m)$.
  Moreover, the distribution is independent of the pairings
  at vertices other than~$v$.
\end{lemma}
\begin{proof}
Take an arbitrary (not necessarily random) pairing
  at every vertex other than~$v$.
  This defines a set of $m/2$ trails that start and finish at $v$ without
  intermediate visits to~$v$.
  Now choosing a random pairing at $v$ combines these $m/2$ trails into induced closed trails, in a way combinatorially equivalent to forming the cycles of the composition of two random fixed-point-free involutions on $m$ letters.
  Lugo has shown this distribution to be $X(m)$~\cite[Prop.\ 8.2]{lugo}.
\end{proof}

\begin{lemma}\label{Utrails}
Consider a $d$-regular graph $G$ and let $U$ be a subset of its vertices.
Let $Y(U)$ be the number of closed trails incident with $U$ induced by a random
pairing.
Then, for constant $\lambda\ge 0$, \[
\E e^{\lambda Y(U)} = d^{\,O(|U|)}.
\]
\end{lemma}
\begin{proof}
For convenience assume $U=\{1,\ldots,|U|\}$.
By Lemma~\ref{Zmprops}, there is a constant
$a=a(\lambda)$ such that $\E e^{\lambda X(d)}\le d^{\,a}$.

First choose all the trails that include vertex~1
by randomly pairing just the vertex-edge incidences
used by those trails.
The number of trails has distribution $X(d)$,
so $\E e^{\lambda Y(\{1\})}\le d^a$.
Next, while avoiding the edges already chosen, randomly choose
all the additional trails that use vertex~2.
The number of additional trails has distribution~$X(m)$,
where $m$ is the number of edges incident with vertex~2 that are not in the trails that that include vertex~1.
Although $m$ depends on the earlier choices, 
  Lemma~\ref{onevertex} says that these two trail counts
  are independent when conditioned on~$m$.
  Therefore, since $\E e^{\lambda X(2j)}\le\E e^{\lambda X(d)}$ for $2j\le d$,
  \[
      \E e^{\lambda Y(\{1,2\})}
      \le \sum_{j=0}^{d/2} \,\Pr(m=2j) \E d^{\,a} \E e^{\lambda X(2j)}
      \le \sum_{j=0}^{d/2} \,\Pr(m=2j) d^{\,2a}
      = d^{\,2a}.
  \]
Continuing in the same manner inductively, we have
$\E e^{\lambda Y(U)} \le d^{\,a|U|}$.
\end{proof}

Recall that $L_k(\Pvec)$ is the number of induced closed
trails with more than $k$ distinct vertices
in $\calT(\Pvec)$ where $k$ is defined in~\eqref{eq:kDef}.

\begin{lemma}\label{bigtrails}
  For any fixed $\lambda\ge 0$, and $k\gg\log d$,
  \[
      \E e^{\lambda L_k(\Pvec)} = e^{o(n)}.
  \]
\end{lemma}
\begin{proof}
  Since $L_k(\Pvec)$ is non-increasing in $k$, we can
  assume $k=o(n)$.
  We will apply Lemma~\ref{Utrails} for $|U|=\lfloor n/k\rfloor$.
  First consider such a subset $U$ of $[n]$ chosen at random.
  A trail using $k$ or more distinct vertices has probability
  $1-(1-|U|/n)^k$ of being hit by~$U$, which is more than half.
  Therefore for any pairing, there is some $U$ of this size
  such that $L_k\le 2 Y(U)$.
  It follows that
  \[
    \E e^{\lambda L_k} \le \sum_{U:\abs{U}=\lfloor n/k\rfloor} \E e^{2\lambda Y(U)}
      \le \binom{n}{\lfloor n/k\rfloor} d^{\,O(n/k)}
      \le (e d k)^{O(n/k)} = e^{o(n)}. \qedhere
  \]
\end{proof}

\section{Short closed trails}\label{S:short}

In this section, we first introduce a general switching theorem from \cite{Sissy}. Then we use it
to establish a tail probability bound for 
the number of short closed trails in $\calT(\Pvec)$, which will be sufficient to get the desired bound for $\E 2^{\frac75 S_k(\Pvec)}$.

\subsection{The switching theorem}\label{s:switchinglemma}

The following theorem and explanation of its use are taken
from~\cite[Theorem 2.1]{Sissy}, where it was proved using~\cite{switchings}.

\begin{thm}\label{switchings}
Let $\varGamma=\varGamma(V,E)$ be a directed multigraph,
and let $\alpha:E\to \mathbb{R}_+$ be a positive weighting
of the edges of $\varGamma$.  Fix a non-empty finite set $C$, whose elements we will
call \textit{colours}, and let $c:E\to C$ be an edge colouring of $\varGamma$  (which need not be proper).
For all $v\in V$ and $c\in C$ denote by $\varGamma_c^-(v)$ the set of
edges of colour~$c$ entering~$v$, and by $\varGamma_c^+(v)$ the
set of edges of colour $c$ leaving~$v$.

We introduce the set of variables $\{N(v) : v\in V\}\cup\{s(e) : e\in E\}$
and consider the following system
of linear inequalities on these variables:
\begin{equation}\label{ineqs}
\left.
\begin{aligned}
   N(v) \ge 0,& \qquad (v\in V) \\
   s(e) \ge 0,& \qquad (e\in E)  \\
   \sum_{e\in\varGamma_c^-(v)} s(e) \le N(v),&
     \qquad(v\in V,\,\, c\in C) \\
   \sum_{e\in\varGamma_c^+(v)} \alpha(e)s(e) \ge N(v),&
     \qquad(v\in V,\,\, c\in C,\,\, \varGamma_c^+(v)\ne\emptyset). 
\end{aligned}
\quad\right\}
\end{equation}
Let $C(v)=\{c\in C :\varGamma_c^-(v)\ne \emptyset \}$ be the
set of colours entering~$v$.
For each $v\in V$ which is not a source, 
let $\lambda_c(v), c\in C(v)$
be positive numbers such that
$\sum_{c\in C(v)}\lambda_c(v)\le 1$.
For each $vw\in E$, define $\hat\alpha(vw)=\alpha(vw)/\lambda_{c(vw)}(w)$.
Extend this function to 
any directed path $D$ by defining $\hat\alpha(D)=\prod_{e\in D}\hat\alpha(e)$.

Now suppose that $Y,Z\subseteq V$ satisfy the following conditions:
\begin{itemize}[noitemsep]
\item[1.] $Z\ne \emptyset$ and $Y\cap Z=\emptyset$;
\item[2.] If $v\in V$ is a sink of $\varGamma$, or if $\hat\alpha(vw)\ge 1$
   for some $vw\in E$, then $v\in Z$.
\end{itemize}
For any $W,W'\subseteq V$, define $\mathcal D(W,W')$ to be the set of 
non-trivial directed paths in $\varGamma$ which start in $W$, 
end in $W'$, and have no internal vertices in $Y\cup Z$.
Then every solution of~\eqref{ineqs} satisfies
\[
  \sum_{v\in Y} N(v) \le
  \frac{\max_{D\in\mathcal D(Y,Z)}\hat\alpha(D)}{1-\max_{D\in\mathcal D(Y,Y)}\hat\alpha(D)}
  \sum_{v\in Z} N(v),
\]
where the maximum over an empty set is taken to be~0.
\end{thm}

We now describe how Theorem~\ref{switchings} can be used for
counting.

Suppose we have a finite set of ``objects'' partitioned into disjoint
classes $\class(v)$, where $v\in V$ for some index set $V$.
Define $N(v)=\card{\class(v)}$ for each $v\in V$. 
Also suppose that for each $c\in C$ we have a relation
$\varPsi_c$ between objects: to be precise, $\varPsi_c$ is
a multiset of ordered pairs $(Q,R)$ of objects.
(We call $\varPsi_c$ a \textit{switching} and usually define it by
some operation that modifies $Q$ to make~$R$.)

Now define an edge-coloured directed multigraph $\varGamma=(V,E)$
with vertex set~$V$, where
$\varGamma$ has a directed edge $vw$ of colour $c$ if and only if
$(Q,R)\in \varPsi_c$ for some $Q\in\class(v)$ and $R\in\class(w)$.
(There is at most one edge of each colour between any pair
of distinct vertices of $\varGamma$.) 
For each $vw\in E$ let
$s'(vw) = \card{\{(Q,R)\in \varPsi_c : Q\in\class(v),R\in\class(w)\}}$,
counting multiplicities, where $c$ is
the colour of~$vw$.

Fix $v\in V$ and $c\in C$ such that $\varGamma_c^+(v)\neq \emptyset$.
Suppose that for any $Q\in\class(v)$ there are at least
$a_c(v)>0$ objects $R$ with $(Q,R)\in \varPsi_c$, counting
multiplicities.  Then
\[
    \sum_{e\in\varGamma_c^+(v)} s'(e)\ge a_c(v)N(v).
\]
Similarly,  for fixed $w\in V$ and $c\in C$, suppose that for every $R\in\class(w)$
there are at most $b_c(w)>0$ objects $Q$ with $(Q,R)\in \varPsi_c$,
counting multiplicities. Then
\[
    \sum_{e\in\varGamma_c^-(w)} s'(e)\le  b_c(w)N(w).
\]
Defining $s(vw)=s'(vw)/b_{c(vw)}(w)$
and $\alpha(vw)=b_{c(vw)}(w)/a_{c(vw)}(v)$ we obtain~\eqref{ineqs}.
Theorem~\ref{switchings} can thus be used to bound the relative values
of $\sum_{v\in Z} \card{\class(v)}$ and $\sum_{v\in Y} \card{\class(v)}$ if
$Y,Z$ satisfy the requirements of the lemma.
Since $\sum_{v\in Z} \card{\class(v)}\le\sum_{v\in V} \card{\class(v)}$,
this also bounds $\sum_{v\in Y} \card{\class(v)}$ relative to
$\sum_{v\in V} \card{\class(v)}$; i.e., it bounds the fraction of
all objects that lie in $\bigcup_{v\in Y} \class(v)$.

\subsection{Switchings for Eulerian partitions}

Let $c_{k,\ell}(G)$ be the number of closed trails of length $\ell$ in $G$
that contain at most $k$ distinct vertices, where $k$ is defined in~\eqref{eq:kDef}.
Let $L = \binom{k}{2}$ denote the maximum 
possible length of such a trail.

\begin{lemma}
\label{lemma:ckl}
  Under the conditions of Theorem~\ref{T:main},
  $c_{k, \ell}(G) \le Ce^{-(\ell+1)}d^{\ell-1}n$
for all $\ell \le L$.
\end{lemma}

\begin{proof}
By the assumptions of Theorem \ref{T:main}, it is true for $\ell \le \lmax$. \
When $\ell>\lmax$, we can bound 
\[
    c_{k,\ell}(G) \le n d^{k-1} k^{\ell-1},
\]
where $n d^{k-1}$ is an upper bound for  choosing 
$k$ distinct vertices that belong to a closed trail (with a root vertex) 
and $k^{\ell-1}$ is a bound for the number of choices to make a trail starting from the root vertex. Since $k \le\lmax/2$ and $k \le \log^2 d$ by the definition of $k$ in \eqref{eq:kDef}, 
and our assumption 
$k \le\lmax/2 \le \ell/2$,
we have
\[
    d^{k-1} k^{\ell-1} 
    \le 
    d^{\ell/2-1} (\log^2 d)^{\ell-1} 
    =o\bigl( e^{-(\ell+1)}d^{\ell-1}\bigr).
\]
Thus, we have established the required bound for $\ell>\lmax$.
\end{proof}

Consider an Eulerian partition $P \in \calP(G)$ and a closed trail $T$
induced by~$P$.
A \textit{$T$-switching} on $P$ modifies $P$ at
each vertex visited by~$T$, as follows.
Consider a vertex $v$ that is visited $t\ge 1$ times by~$T$. 
Let $(e_1,e'_1),\ldots,(e_t,e'_t)$ be orderings of
the edge pairs at~$v$ which are traversed by~$T$.
 Choose $t$ additional
 ordered pairs
$(e_{t+1},e'_{t+1}),\ldots,(e_{2t},e'_{2t})$
of edges incident to $v$ 
such that all $\{(e_i,e'_i):
i \in [2t]\}$
are distinct and belong to trails induced by $P$.
Now replace $\{e_1,e'_1\},\ldots,\{e_{2t},e'_{2t}\}$ by
$\{e_1,e_{t+1}\},\ldots,\{e_t,e_{2t}\}$ and
$\{e'_1,e'_{t+1}\},\ldots,\{e'_t,e'_{2t}\}$.
Performing this operation at each vertex visited by $T$
produces a new Eulerian partition~$P'$.
Two $T$-switchings on $P$ are considered the same if the resulting partition~$P'$
is the same; however $T$-switchings for different $T$ are always considered different.

Given $P'$ and $T$, we can recover $P$ using an \textit{inverse $T$-switching}.
As before, consider a vertex $v$ that is visited $t\ge 1$ times by~$T$
and let $\{e_1,e'_1\},\ldots,\{e_t,e'_t\}$ be
the edge pairs at~$v$ which are traversed by~$T$.
If any two of $e_1,\ldots,e_t,e'_1,\ldots,e'_t$ are paired in $P'$,
no inverse $T$-switching can be performed.
Otherwise, suppose $\{e_1,e_{t+1}\},\ldots,\{e_t,e_{2t}\}$ and
$\{e'_1,e'_{t+1}\},\ldots,\{e'_t,e'_{2t}\}$ are pairs in~$P'$.
Replacing them by 
$\{e_1,e'_1\},\ldots,\{e_{2t},e'_{2t}\}$ recovers~$P$.

An illustration of how a $T$-switching acts at one
vertex is shown in Figure~\ref{fig:switching}.
Given $P$ and $T$ induced by $P$, there can be many $T$-switchings but
each produces a different~$P'$.
Conversely, given $P'$ and $T$, there is at most one $P$ and one
$T$-switching on $P$ that produces~$P'$.

If $T$ has length~$\ell$, a $T$-switching is also called an
\textit{$\ell$-switching}.

\begin{figure}[ht]

\centering
\begin{minipage}{.4\textwidth}
\begin{tikzpicture}

\node[circle, draw, inner sep=.8pt, fill] (centre) at (0, 0) {};
\node[circle] (P) at (0, -3.2) {$P$
at vertex $v$};
\node[circle] (v) at (0, .6) {$v$};
\node[inner sep=0] (c0) at (.08, .08) {};
\node[inner sep=0] (c1) at (-.08, .08) {};
\node[inner sep=0] (c2) at (-.16, -.16) {};
\node[inner sep=0] (c4) at (-.08, -.08) {};
\node[inner sep=0] (c3) at (.08, -.08) {};
\foreach \angle in {0,36,...,324}
{
    \node (P\angle) at (\angle:3) {};
}

\node (e1) at (36+9:2) {$e_1$};
\node (e1') at (72-9:2) {$e_1'$};

\node (e2) at (216+9:2) {$e_2$};
\node (e2') at (252-10:2) {$e_2'$};

\node (e3) at (324+9:2) {$e_3$};
\node (e3') at (0-9:2) {$e_3'$};

\node (e4) at (144-9:2) {$e_4$};
\node (e4') at (108+9:2) {$e_4'$};

\draw[line width=.5mm, blue, rounded corners=4pt] (P36) -- (c0.center) -- (P72);

\draw[line width=.5mm, blue, rounded corners=4pt] (P216) -- (c2.center) -- (P252);

\draw[line width=.5mm, red, rounded corners=4pt] (P108) -- (c1.center) -- (P144);

\draw[line width=.5mm, red, rounded corners=4pt] (P324) -- (c3.center) -- (P0);

\draw[densely dotted, line width=.5mm, red, rounded corners=4pt] (P180) -- (c4.center) -- (P288);
\end{tikzpicture}
\end{minipage}
\begin{minipage}{.4\textwidth}
\begin{tikzpicture}
\node[circle, draw, inner sep=.8pt, fill] (centre) at (0, 0) {};
\node[circle] (P) at (0, -3.2) {$P'$
at vertex $v$};
\node[circle] (v) at (0, .6) {$v$};
\node[inner sep=0] (c0) at (.1, 0.15) {};
\node[inner sep=0] (c1) at (0.12, 0) {};

\node[inner sep=0] (c2) at (-.1, .15) {};

\node[inner sep=0] (c3) at (-.2, 0) {};

\node[inner sep=0] (c4) at (-.1, -.2) {};
\foreach \angle in {0,36,...,324}
{
    \node (P\angle) at (\angle:3) {};
}

\node (e1) at (36+9:2) {$e_1$};
\node (e1') at (72-9:2) {$e_1'$};

\node (e2) at (216+9:2) {$e_2$};
\node (e2') at (252-10:2) {$e_2'$};

\node (e3) at (324+9:2) {$e_3$};
\node (e3') at (0-9:2) {$e_3'$};

\node (e4) at (144-9:2) {$e_4$};
\node (e4') at (108+9:2) {$e_4'$};

\draw[line width=.5mm, rounded corners=4pt] (P36) -- (c1.center) -- (P324);

\draw[line width=.5mm, rounded corners=4pt] (P72) -- (c0.center) -- (P0);

\draw[line width=.5mm, rounded corners=4pt] (P108) -- (c2.center) -- (P252);

\draw[line width=.5mm, rounded corners=4pt] (P144) -- (c3.center) -- (P216);

\draw[densely dotted, line width=.5mm, rounded corners=4pt] (P180) -- (c4.center) -- (P288);
\end{tikzpicture}
\end{minipage}
\caption{Example of the action of a switching at one vertex.
Simultaneous similar actions at all vertices visited by the trail $T$ transform
Eulerian partition $P$ into $P'$.
Pairs $(e_1, e_1')$ and $(e_2, e_2')$ are in $T$.
The dotted edge pair in $P$ is not involved in the switching and remains in $P'$.}
\label{fig:switching}
\end{figure}
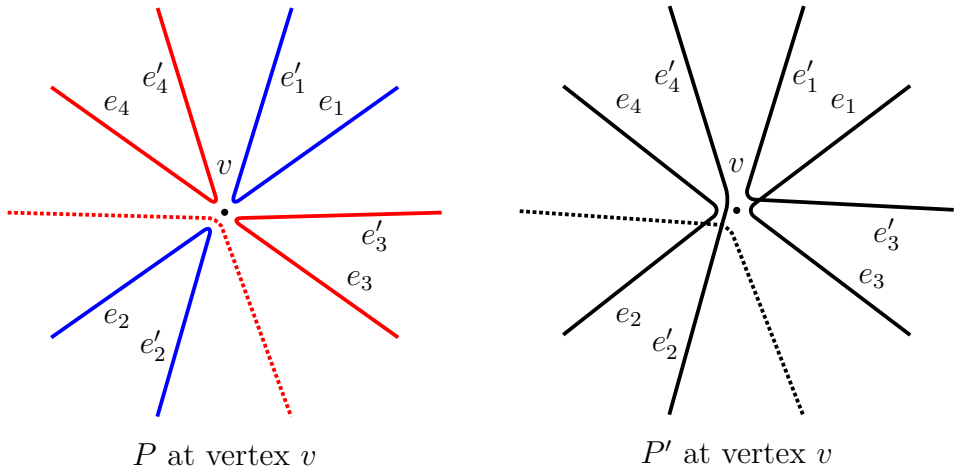

\medskip
For $\mvec=(m_3,\ldots,m_{L})$, let 
 $
   \class(\mvec)
$ 
denote the set of Eulerian partitions $P$ such that, for any  $3\le \ell\le L$, 
$P$ induces exactly $m_\ell$ closed trails of length~$\ell$ and
at most distinct $k$ vertices.
The edge-coloured multidigraph $\varGamma$ has vertices
$\{\mvec\st\class(\mvec)\ne\emptyset\}$
and $L-2$ colours indexed by 
$\{ 3, 4, \ldots, L\}$.
The directed edges of $\varGamma$ with colour $\ell$ are those pairs
$(\mvec,\mvec')$ such that $m_\ell\ge \|\mvec\|_1/L$,
and   
there exists an $\ell$-switching on some $P \in \class(\mvec)$
that yields an Eulerian partition 
$P' \in \class(\mvec')$. Here we used the standard notion of the $1$-norm:
\[
    \|\mvec\|_1 := m_3+\cdots+m_{L}.
\]

\begin{lemma}\label{L:switchcount}
Let $(\mvec, \mvec')$
be an edge of $\varGamma$ with colour $\ell$.
Then, for any 
$P \in \class(\mvec)$,
the number of 
$\ell$-switchings on $P$
is at least $\|\mvec\|_1(d-2L)^{\ell}/L$. Furthermore, for any 
$P' \in \class(\mvec')$,  there are at most $c_{k, \ell}$
$\ell$-switchings that can produce $P'$.
\end{lemma}

\begin{proof}
For $P \in \class(\mvec)$,
it induces $m_\ell\ge\|\mvec\|_1/L$ closed trails of length $\ell$ because there is an edge leaving $\mvec$  of colour $\ell$.
Suppose that a vertex is visited 
$1\le t \le \ell/2$ times by a closed trail $T$ of length $\ell$ induced by $P$,
then we need to choose a sequence of $t$ additional ordered pairs at this 
vertex and there are $d-2t$ edges to choose from.
To choose the $j$-th new pair for $j \in [t]$,
there are $d-2t-2(j-1)
\ge d-4t \ge d - 2L$ ways.  Overall, we have at least $(d-2L)^t$ choices for this vertex. Taking the product over all vertices of $T$, we conclude that the number of 
$\ell$-switchings on pairings in~$\class(\mvec)$
is at least
$m_\ell (d-2L)^{\ell}
\ge
\|\mvec\|_1(d-2L)^{\ell}/L$.

To establish the second part, from the description of an inverse $T$-switching
we know that for each $P'\in\class(\mvec')$ and each closed trail $T$ in $G$,
there is at most one partition $P$ and one $T$-switching on $P$ that
gives~$P'$.
Thus the number of $\ell$-switchings whose result is $P'$ is
at most~$c_{k,\ell}$.
\end{proof}

Let $\class_m$ be the set of $\mvec\in V(\varGamma)$ that correspond to Eulerian partitions that induce
exactly $m$ closed trails with at most $k$ distinct vertices of $G$.
Let
$Y=\bigcup_{m > M}\class_m$
and 
$Z=\bigcup_{m\le M_0}\class_m$, where 
\begin{align}\label{eq:M0def}
    M>M_0:=2CnL^2/d,
\end{align}
and $C$ is the constant from Theorem~\ref{T:main}.

For any $m>M_0$, every vertex in $\class_m$ has at 
least one edge leaving, by the pigeonhole principle,  therefore all sinks are in $Z$.
We define $N(\mvec) = |\class(\mvec)|$ and $\alpha(e)$, and
$s(e)$ for all $e\in E(\varGamma)$
as described in Section \ref{s:switchinglemma}
after Theorem \ref{switchings}.
Since $L = \binom{k}{2} \le \log^4d$
by the definition of $k$
in \eqref{eq:kDef},
Lemma \ref{L:switchcount} gives
\[
    \alpha(e) \le \frac{c_{k, \ell}L}{m(d-2L)^\ell}
     = \frac{c_{k, \ell}L}{m d^{\ell}} e^{O(L^2/d)}
     \le \frac{2c_{k, \ell}L}{m d^{\ell}},
\]
for any edge $e$ with colour $\ell$ leaving a vertex in $\class_m$. 
Choosing $\lambda_\ell(\mvec)=L^{-1}$ for each vertex $\mvec$ of $\varGamma$ and all $3\le\ell\le L$,
we also get that 
\begin{equation}\label{eq:alphahat}
    \hat\alpha(e)\le \frac{2c_{k, \ell}L^2}{d^{\ell}m}.
\end{equation}

Next,  by applying Theorem~\ref{switchings}, we can bound the number of Eulerian partitions containing too many short trails.
\begin{lemma}\label{L:YZ}
  Under the conditions of Theorem~\ref{T:main},
  $Z\ne\emptyset$ and
  \[
  \abs{Y}   \le  2e^{-M+M_0}  \abs{Z}.
  \]
\end{lemma}
\begin{proof}
Let $e =(\mvec,\mvec')$ be an edge of $\varGamma$ with $\|\mvec\|_1> M_0$.
Combining Lemma \ref{lemma:ckl}
and the bound on 
$\hat\alpha(e)$ in
\eqref{eq:alphahat}, we obtain 
\begin{equation}\label{hatalpha}
    \hat\alpha(e) 
    \le
    \frac{2Cn L^2}{d M_0}
    e^{-(\ell+1)}
     = e^{-(\ell+1)}
     < \dfrac12,
\end{equation}
in view of the definition of $M_0$ in \eqref{eq:M0def}.
If $Z=\emptyset$ then the fact that $\alpha(\mvec,\mvec')<1$ whenever
$\mvec\notin Z$ would imply that the total number of switchings out of
vertices of $\varGamma$ exceeds the total number of switchings into those
vertices. This impossibility shows that $Z\ne\emptyset$.
Thus, all assumptions of Theorem \ref{switchings}  hold as explained in Section \ref{s:switchinglemma}.  

Recall the definition of $\mathcal D(W,W')$
from Theorem \ref{s:switchinglemma}.
Note that \eqref{hatalpha} implies  
\[
\max_{D\in\mathcal D(Y,Y)}\hat\alpha(D) \le \dfrac12,
\]
since any such directed path contains at least one edge with its starting vertex outside of~$Z$.
Next, consider any directed path $D\in \mathcal{D}(Y,Z)$.
Since any $\ell$-switching can decrease the number of short closed
trails by at most $\ell+1$, using \eqref{hatalpha}, we find that 
\[
    \hat\alpha(D) = 
    \prod_{e\in D}\hat\alpha(e) 
    \le 
    \exp\biggl(-\sum_{e\in D} (\ell_e+1)\biggr)
    \le e^{-M+M_0},
\]
where edge $e$ represents an $\ell_e$-switching.
Applying Theorem \ref{switchings}, we get the stated bound.
\end{proof}

Recall that $S_k(\Pvec)$ is the number of closed trails
with at most $k$ distinct vertices
induced by the random Eulerian partition $\Pvec$,
where $k$ is defined in~\eqref{eq:kDef}.
\begin{lemma}
\label{smalltrails}
Under the conditions of Theorem~\ref{T:main},
\[
  \E 2^{\frac{7}{5} S_k(\Pvec)} = e^{o(n)}.   
\]
\end{lemma}
\begin{proof}
Recalling the definition of the set $Y$
before \eqref{eq:M0def},
using Lemma \ref{L:YZ}, we find that 
\[
    \Pr(S_k(\Pvec)>M) = \frac{\abs{Y}}{|\calP(G)|}
    \le \frac{\abs{Y}}{\abs{Z} } \le  2 e^{-M+M_0}.
\]
Let $\lambda = \frac75 \log 2  \approx 0.97$.
Then
\[ 
\E 2^{\frac75 S_k(\Pvec)} =
\E e^{\lambda S_k(\Pvec)}
=
\E\left[ 1 + \lambda \int_{0}^{\infty} e^{\lambda t}\,\mathbf{1}_{\{S_k(\Pvec)>t\}}\,dt \right]
= 1 + 
\lambda\int_{0}^\infty e^{\lambda x}\Pr(S_k(\Pvec)>x)\,dx,
\]
so we obtain that 
\[
   \E e^{ \lambda S_k(\Pvec)}
   \le 1 + \int_{0}^{M_0} \lambda  e^{\lambda x}\,dx
      + 2\int_{M_0}^\infty e^{\lambda x-(x-M_0)}\, dx 
     = e^{\lambda M_0} +\Dfrac{2e^{\lambda M_0}}{1-\lambda}.
\]
Recalling the definition of $M_0$ from \eqref{eq:M0def} 
and that $L = O(\log^4 d)$, 
we have $e^{\lambda M_0} = e^{o(n)}$. This completes the proof.
\end{proof}

\section{Acknowledgements}
This research was initiated at the  Combinatorics of McKay and Wormald workshop at the
Mathematical Research Institute MATRIX (June 2025).


\begin{thebibliography}{10}
\itemsep=0.1ex

\bibitem{baxter1969f}
R.\,J. Baxter,
  F model on a triangular lattice,
   \textit{J. Math.\ Physics}, \textbf{10} (1969) 1211--1216.

\bibitem{Sissy}
C. Greenhill and B.\,D. McKay,
Asymptotic enumeration of sparse multigraphs with
given degrees,
\textit{SIAM J. Discrete Math.},
\textbf{27} (2013) 2064--2089.


\bibitem{switchings}
M. Hasheminezhad and B.\,D. McKay,
Combinatorial estimates by the switching method,
\textit{Contemporary Math.}, \textbf{531} (2010) 209--221.

\bibitem{Hoeffding} W.~Hoeffding,
 Probability Inequalities for sums of bounded random variables,
 \textit{J. Amer. Stat. Ass.}, \textbf{58} (1963) 13--30. 


\bibitem{IMZ}
M. Isaev, B.\,D. McKay and R.-R. Zhang,
Cumulant expansion for counting Eulerian orientations,
\textit{J. Combin.\ Th., Ser.\ B}, \textbf{172} (2025) 263--314.

\bibitem{IMZ1}
M. Isaev, B.\,D. McKay and R.-R. Zhang,
Correlation between residual entropy and spanning tree entropy of ice-type models on graphs.
\textit{Ann. Inst. Henri Poincar{\'e} D, Comb. Phys. Interact.}, (2025).


\bibitem{LasVergnas1990}
M. Las Vergnas,
An upper bound for the number of Eulerian orientations
of a regular graph,
\textit{Combinatorica}, \textbf{10} (1990) 61--65.

\bibitem{LieHexagonalMonolayer}
D-Z. Li, W-J. Huang, Y. Yao and X-B. Yang,
Exact results for the residual entropy of ice hexagonal monolayer,
\textit{Phys.\ Rev.\ E}, \textbf{107} (2023) \#054121.

\bibitem{lieb1967residual}
E.\,H. Lieb,
  Residual entropy of square ice,
   \textit{Physical Review}, \textbf{162} (1967) 162--172.

\bibitem{lieb1972two}
E.\,H. Lieb and F.\,Y. Wu,
  Two-dimensional ferroelectric models,
  in \textit{Phase Transitions and Critical Phenomena},
  C. Domb and M. Green eds.,
  vol. 1, Academic Press, 1972, 331--490. 

\bibitem{lugo}
  M. Lugo,
  The cycle structure of compositions of random involutions,
  preprint (2009), arXiv:0911.3604.


\bibitem{mckayeigens}
B. D. McKay,
The expected eigenvalue distribution of a large regular graph,
\textit{Linear Algebra Appl.}, \textbf{40} (1981) 203--216.


\bibitem{Merris1998}
R. Merris,
Laplacian graph eigenvectors,
\textit{Linear Algebra Appl.}, \textbf{278} (1998) 221--236.


\bibitem{pauling1935structure}
L. Pauling,
  The structure and entropy of ice and of other crystals with some
  randomness of atomic arrangement,
   \textit{J. Amer.\ Chem.\ Soc.}, \textbf{57} (1935) 2680--2684.


\end{thebibliography}
\end{document}